\documentclass[11pt,eqright]{amsart}
\usepackage{amssymb,amsmath}
\usepackage{bm}
\usepackage{amsfonts}
\usepackage{epsf}
\usepackage{graphicx}
\newtheorem{theorem}{Theorem}[section]
\newtheorem{lemma}[theorem]{Lemma}

\newtheorem{remark}[theorem]{Remark}

\theoremstyle{definition}
\usepackage[titletoc]{appendix}
\usepackage{subfigure}
\numberwithin{equation}{section}

\newcommand{\bc}{\begin{center}}
\newcommand{\ec}{\end{center}}
\newcommand{\be}{\begin{eqnarray}}
\newcommand{\ee}{\end{eqnarray}}
\newcommand{\nn}{\nonumber}
\newcommand{\ben}{\begin{eqnarray*}}
\newcommand{\een}{\end{eqnarray*}}

\newcommand{\Om}{\Omega}

\newcommand{\lam}{\lambda}

\newcommand{\na}{\nabla}

\def\x{\times}

\def\na{\nabla}

\def\cE{\mathcal{E}}

\def\cT{\mathcal{T}}
\def\R{\mathbb{R}}

\def\bsi{\boldsymbol{\sigma}}
\def\bf{\boldsymbol{f}}
\def\bu{\boldsymbol{u}}
\def\bv{\boldsymbol{v}}
\def\btau{\boldsymbol{\tau}}

\def\div{\operatorname{div}}

\DeclareMathOperator{\sspan}{span}

\DeclareMathOperator{\id}{id}
\DeclareMathOperator{\tr}{tr}

\DeclareMathOperator{\dev}{dev}

\newcommand{\reffig}[1]{Figure \ref{#1}}

\title[]
{\small The enriched Crouzeix--Raviart elements are equivalent to  the Raviart--Thomas elements}
\author[J.~Hu]
{Jun Hu$^\ast$}
\address{$^\ast$ LMAM and School of Mathematical Sciences,
 Peking University, Beijing 100871, P. R. China}
\email{hujun@math.pku.edu.cn}
\author[R. Ma]{Rui Ma$^\dagger$}
\address{$^\dagger$ LMAM and School of Mathematical Sciences,
 Peking University, Beijing 100871, P. R. China}
\email{maruipku@gmail.com}
\thanks{The  first author was supported by  the NSFC Project 11271035 and  by  the NSFC Key Project 11031006.}

\begin{document}

\newpage
\begin{abstract}
For both the Poisson model problem and the Stokes problem in any dimension,
  this paper proves that  the  enriched  Crouzeix--Raviart elements
 are  actually identical to the first order Raviart--Thomas elements in the sense that they produce the
 same discrete stresses.  This result improves the previous  result in literature which, for two dimensions,  states that
  the piecewise constant projection of the stress by the first order Raviart--Thomas element is equal to
  that by the Crouzeix--Raviart element. For the eigenvalue problem of Laplace operator, this paper proves that the error of the enriched Crouzeix--Raviart element is equivalent to that of the Raviart--Thomas element up to higher order terms.
\end{abstract}

\maketitle

\section{introduction}
The aim of this paper is to prove the enriched Crouzeix--Raviart (ECR hereafter) elements by Hu, Huang and Lin \cite{HuHuangLin2010} are equivalent to the first order Raviart--Thomas elements (RT hereafter). The first main result proves that ECR elements are identical to RT  elements for both the Poisson and Stokes problems in any dimension.
     More precisely, for the Poisson problem imposed a piecewise constant right--hand  function $f$, it is proved that
     \begin{equation}\label{ECR-RT}
     \sigma_{\rm RT}=\na_{\rm NC} u_{\rm ECR} \text{ and }u_{\rm RT}=\Pi_0u_{\rm ECR},
     \end{equation}
     where $u_{\rm ECR}$ and $(\sigma_{\rm RT},u_{\rm RT})$ denote the finite element solutions by the ECR and RT elements, respectively; while for the Stokes problem imposed a piecewise constant right--hand  function $\bf$, it is established that
     \begin{equation}\label{ECR-RT-Pseudo}
     \bsi_{\rm RT}=\na_{\rm NC} \bu_{\rm ECR}+p_{\rm ECR}\id \text{ and }\bu_{\rm RT}=\Pi_0\bu_{\rm ECR}+\mathcal{L}u_{\rm ECR},
     \end{equation}
      where $(\bu_{\rm ECR},p_{\rm ECR})$ and $(\bsi_{\rm RT},\bu_{\rm RT})$ denote the finite element solutions by the ECR and RT elements, respectively. Herein and throughout this paper, $\Pi_0$ denotes the piecewise  constant $L^2$ projection with respect
      to  a shape--regular partition $\cT$ of $\Omega$  consisting of $n$-simplices, and $\mathcal{L}$ is some linear operator.
       The second main result proves that for the eigenvalue problem of Laplace operator
       \begin{equation}\label{ECR-RTeigen}
         \|\nabla_{\rm NC}( u-u_{\rm ECR})\|=\|\nabla u-\sigma_{\rm RT}\|+h.o.t.
       \end{equation}
where the constants involved in the high order term depend on the corresponding eigenvalue. Throughout this paper, $\|v\|$ denotes $(v,v)^{1/2}_{L^2(\Omega)}$, for any $v\in L^2(\Omega)$.  See the next two sections for more details of the notations.

     The history perspective  justifies
      the novelty of both \eqref{ECR-RT} and \eqref{ECR-RT-Pseudo}. For general right--hand function $f$
      \begin{equation}\label{1.3}
        \|\nabla_{\rm NC}(u-u_{ECR})\|=\|\nabla u-\sigma_{\rm RT}\|
      \end{equation}
     hold up to data oscillation. Indeed, it is the first time that the RT  elements
      are proved in such a direct and simple way   to be  identical  to  nonconforming finite elements in any dimension while the previous
      results  state some relations between the two dimensional Crouzeix--Raviart (CR hereafter) and RT elements; see below  and also \cite{Arnold,CarstensenGallistlSchedensack2013,Marini1985}  for  more details.  These results imply that
       the RT element can not be equivalent to the CR element in general, which gives a negative answer to an open problem in \cite{CarstensenPeterseimSchedensack2012}.

      The study on  the relations between nonconforming finite elements and mixed finite elements can
 date back to the pioneer and remarkable work  by Arnold and Brezzi in 1985 \cite{Arnold}. In particular,  for the two dimensional biharmonic equation,
  it was proved  that  the first order  Hellan--Herrmann--Johnson element \cite{Hellan1967,Herrmann1967,Johnson1973}
  is identical to the modified Morley element which differs from the usual Morley element  \cite{Morley68} only by presence
   of the interpolation operator in the right--hand side; while for the two dimensional Poisson problem, it was shown that
    the $L^2$ projection onto the first order RT element space of the stress by
 the CR element, enriched by piecewise cubic bubbles, is identical to the stress by
  the RT element.  By proposing initially a projection finite element method,
    Arbogast and Chen \cite{ArbogastChen1993} generalized successfully the idea of  \cite{Arnold} to most mixed methods
    of  more general  second order elliptic problems in both two and three dimensions. In particular, they showed that
    most mixed methods can be implemented by solving  projected nonconforming methods with
    symmetric positive definite stiff matrixes, and that stresses by mixed methods are $L^2$ projections of those by nonconforming
      methods.  Let $\sigma_{\rm RT}$ be the discrete stress by the RT element, and $u_{\rm CR}$ be
  the displacement by the CR element of the two dimensional Poisson equation.  Suppose that $f$ is a piecewise constant function with  respect to
  $\cT$.  Marini further explored the relation  between the RT and CR elements of \cite{Arnold} to derive the following relation \cite{Marini1985}:
  \begin{equation}\label{Marini}
  \sigma_{\rm RT}|_K=\nabla u_{\rm CR}|_K-\frac{f_K}{2}(x-{\rm mid }(K))\quad  x\in K \text{ for any }K\in\cT,
  \end{equation}
  where $f_K:=f|_K$ denotes the  restriction on $K$ of $f$ and ${\rm mid}(K)$ denotes the centroid of $K$. This important identity was exploited by Brenner  \cite{Brenner1999} to
  design an optimal multigrid method for the RT element,  and by Carstensen and Hoppe to
  establish, for  the first time,  quasi--orthogonality and consequently convergence of both the adaptive RT and CR elements in \cite{CarstensenHoppe2006MathComp} and
  \cite{CarstensenHoppe2006NumerMath}, respectively.  For the two dimensional Stokes equation, a similar identity was first accomplished
  in \cite{CarstensenGallistlSchedensack2013}:
  \begin{equation}\label{CGS}
  \begin{split}
  &\bsi_{\rm RT}|_K=\nabla \bu_{\rm CR}-\frac{\bf_K}{2}\otimes(x-{\rm mid}(K))+p_{\rm CR}|_K\id,\\
  &\bu_{\rm RT}|_K=\Pi_0\bu_{\rm CR}+\frac{1}{4}\Pi_0(\dev(\bf_K\otimes(x-{\rm mid}(K))))(x-{\rm mid}(K)),
  \end{split}
  \end{equation}
  $x\in K$ for any $K\in\cT$. Here $(\bu_{\rm CR}, p_{\rm CR})$ and $(\bsi_{\rm RT}, \bu_{\rm RT})$ are finite element
   solutions by the CR and RT  elements, respectively, and $\bf_K$ is the restriction on $K$
   of the piecewise constant function $\bf$.
     Given two vectors $a\in \R^2$ and $b\in \R^2$, $a\otimes b: =ab^{T}$ defines a $2\times 2$  matrix
    of rank one.   See also \cite{HuXu2007} for a similar relation  between the CR and RT elements for the two dimensional Stokes--like
    problems.
    Such a beautiful identity is also used to  prove convergence and optimality of
    the adaptive pseudostress method in \cite{CarstensenGallistlSchedensack2013}.

  There is another direction for the study on the relations between  nonconforming finite  elements and mixed finite elements, which may
   start with the remarkable work by Braess \cite{Braess2009}. A  recent paper on the two dimensional Poisson model problem due to  Carstensen,   Peterseim,
   and Schedensack  \cite{CarstensenPeterseimSchedensack2012} states more general and profound comparison results of mixed, nonconforming and conforming finite element methods
   \begin{equation}
   \|\na u-\sigma_{\rm RT}\|\leq C \|\na_{\rm NC}(u-u_{\rm CR})\| \approx \|\na (u-u_{\rm C})\|,
   \end{equation}
   hold up to data oscillation and up to  mesh-size independent generic multiplicative constants,
      where $C$ is a generic constant independent of the meshsize,  and $u_{\rm C}$ is the finite  element solution by the
    conforming Courant element. See \cite{Gudi(2010),MaoShi(2010)} for some relevant results in this direction. By a numerical counterexample,
     it was also demonstrated in  \cite{CarstensenPeterseimSchedensack2012} that the converse estimate
     \begin{equation}\label{eqRE}
     \|\na_{\rm NC}(u-u_{\rm CR})\|\leq C\|\na u-\sigma_{\rm RT}\|
     \end{equation}
     does not hold. In Subsection 3.3, we give an example where the right hand side of the above inequality vanishes while the left hand side
      is nonzero, which implies that the converse of \eqref{eqRE} is not valid.

This paper is organised as follows.
Section 2 presents the Poisson equation, Stokes equation and their mixed formulations. This section also introduces the ECR and RT elements. Section 3 proves the equivalence between the ECR and RT elements for the Poisson equation and Stokes equation respectively. Section 4  proves the equivalence between the ECR and RT elements for the eigenvalue problem of  Laplace operator. Section 5 shows some numerical results by ECR elements. In the end, the appendix presents the basis functions and convergence analysis of ECR elements.

       \section{Poisson equation, stokes equation and nonconforming finite element methods}
 We present the Poisson equation, stokes equation and their nonconforming finite element methods in this section.  Throughout this paper, let $\Omega\subset \R^n$ denote a bounded domain, which, for the sake of simplicity, we suppose to be a polytope.
       \subsection{The Poisson equation}
Given $f\in L^2(\Om, \R)$,  the Poisson model  problem
finds $u\in H^1_0(\Om, \R)$ such that
\begin{equation}\label{Poisson}
(\na u, \na v)=(f,v)\quad \text{ for all }v\in H^1_0(\Omega, \R).
\end{equation}
By introducing an auxiliary variable $\sigma:=\nabla u$,  the problem can be formulated as the following equivalent
mixed problem which seeks $(\sigma, u)\in H(\div,\Omega,\R^n)\times L^2(\Omega, \R)$ such that
 \begin{equation}\label{MixedPoisson}
\begin{split}
 &(\sigma, \tau)+(u, \div \tau)=0\quad \text{ for any }\tau\in H(\div,\Omega, \R^n),\\
 &(\div \sigma, v)=(-f, v)\quad \text{ for any }v\in L^2(\Omega, \R).
\end{split}
 \end{equation}

 \subsection{The stokes equation} Given $\bf\in L^2(\Om, \R^n)$,
  the Stokes problem models the motion of incompressible fluids occupying $\Omega$
   which finds $(\bu, p)\in H_0^1(\Om, \R^n)\x L_0^2(\Om, \R):=\{q\in L^2(\Om, \R), \int_{\Om}qdx=0\}$ such that
\begin{equation}\label{stokes}
\begin{split}
&(\na \bu,\na \bv)+(\div\bv,p)=(\bf,\bv)\quad \text{ for any }\bv\in H_0^1(\Om, \R^n)\,,\\
&(\div\bu, q)=0 \text{ for any }q\in L_0^2(\Om, \R).
\end{split}
\end{equation}
where $\bu$ and $p$ are the velocity and pressure of the flow, respectively.   Given any $ n\times n$ matrix--valued function $\btau$,
  its  divergence $\div\btau $ is defined as
  $$
  \div\btau:=\begin{pmatrix}
  \sum\limits_{j=1}^n\frac{\partial \tau_{1,j}}{\partial x_j}\\
  \vdots\\
  \sum\limits_{j=1}^n\frac{\partial \tau_{k,j}}{\partial x_j}\\
  \vdots\\
  \sum\limits_{j=1}^n\frac{\partial \tau_{n,j}}{\partial x_j}
  \end{pmatrix},
  $$
 while its trace  reads
 $$
 \tr \btau:=\sum\limits_{i=1}^n\tau_{ii}.
 $$
 Let $\id\in \R^{n\times n}$ be the $n\times n $ identity matrix. This allows to define the deviatoric part of $\btau$ as
 $$
 \dev \btau:=\btau-1/n\  \tr(\btau)\id.
 $$
The definition of the pseudostress $\bsi:=\na \bu+p\id$ yields the equivalent pseudostress formulation \cite{ArnoldFalk1988,CaiLeeWang2004,CaiWang,CarstensenGallistlSchedensack2013,CarstensenKimPark2011,HuXu2007}:
 Find $(\bsi, \bu)\in \hat{H}(\div, \Om, \R^{n\times n})\x L^2(\Om, \R^n)$ such that
\begin{equation}\label{Mixedstokes}
\begin{split}
&(\dev  \bsi, \dev\btau )+(\bu, \div\btau)=0 \quad \text{ for any }\btau \in \hat{H}(\div, \Om, \R^{n\times n})\,,\\
&(\div\bsi, \bv )=-(\bf, \bv) \text{ for any }\bv\in L^2(\Om, \R^n).
\end{split}
\end{equation}
Herein and throughout this paper, the space $\hat{H}(\div, \Om, \R^{n\times n})$ is defined as
$$
\hat{H}(\div, \Om, \R^{n\times n}):=\{\btau\in H(\div, \Om, \R^{n\times n}):\int_{\Omega}\tr \btau dx=0\}.
$$
\subsection{Triangulations}
 The simplest nonconforming finite elements for both Problem \eqref{Poisson} and Problem \eqref{stokes}
  are  the CR elements proposed in \cite{CrouzeixRaviart(1973)}
 while the simplest mixed finite elements for Problem \eqref{MixedPoisson} and Problem \eqref{Mixedstokes} are
  the first order RT element due to \cite{RaviartThomas1977} and \cite{ArnoldFalk1988,CaiLeeWang2004,CaiWang,CarstensenGallistlSchedensack2013,CarstensenKimPark2011,HuXu2007}, respectively. Suppose that $\overline{\Omega}$ is covered exactly by shape--regular partitions $\cT$  consisting of $n$-simplices  in
$n$ dimensions.  Let  $\cE$ denote the set of  all
$n-1$ dimensional subsimplices of $\cT$,   and  $\cE(\Omega)$ denote the set of
all the $n-1$ dimensional interior subsimplices, and  $\cE(\partial \Omega)$ denote  the set of
all the $n-1$ dimensional boundary subsimplices.  Given $E\in \cE$, let $\nu_E$ be  unit normal vector
and  $[\cdot]$ be jumps of piecewise functions over $E$, namely
$$
[v]:=v|_{K^+}-v|_{K^-}
$$
for piecewise functions $v$ and any two elements $K^+$ and $K^-$ which share the common sub-simplice $E$. Note that
$[\cdot]$  becomes  traces of functions  on $E$ for  boundary  sub-simplices
$E$.

\subsection{The enriched Crouzeix--Raviart elements for both  the Poisson and Stokes equations}
 Given $\omega\subset\Omega$ and an integer $m\geq 0$, let $P_m(\omega)$ denote the space of polynomials of degree $\leq m$ over $\omega$.
 The Crouzeix-Raviart element space $V_{\rm CR}$ over $\cT$ is defined as
\begin{equation}
 V_{\rm CR}:=\begin{array}[t]{l}\big\{v\in L^2(\Om,\R):
 v|_{K}\in P_1(K) \text{ for each }K\in \cT,
 \int_E[v]dE=0,\\
 \text{ for all } E\in\cE(\Omega)\,,
 \text{ and }\int_E vdE=0 \text{ for all   $E\in\cE(\partial\Omega)$ }
 \big\}\,.
 \end{array}\nn
\end{equation}
To obtain a nonconforming finite element method that is able to  produce  lower bounds
 of eigenvalues of  second order elliptic operators,  it was proposed in \cite{HuHuangLin2010}
 to enrich the shape function space $P_1(K)$ by $\sspan\{\sum\limits_{i=1}^nx_i^2\}$ on each
element. This leads to the following shape function space
 \begin{equation}\label{ECRfunctionspace}
{\rm ECR}(K):=P_1(K)+\sspan\Big\{\sum\limits_{i=1}^nx_i^2\Big\}\quad \text{ for any
}K\in\cT.
 \end{equation}
  The enriched
 Crouzeix-Raviart element space $V_{\rm ECR}$ is then defined by
\begin{equation}\label{ECR}
 V_{\rm ECR}:=\begin{array}[t]{l}\big\{v\in L^2(\Om,\R):
 v|_{K}\in {\rm ECR}(K) \text{ for each }K\in \cT,
 \int_E[v]dE=0,\\
 \text{ for all } E\in\cE(\Omega)\,,
 \text{ and }\int_E vdE=0 \text{ for all   $E\in\cE(\partial\Omega)$ }
 \big\}\,.
 \end{array}\nn
\end{equation}
The enriched  Crouzeix--Raviart element method of Problem \eqref{Poisson} finds  $u_{\rm ECR}\in V_{\rm ECR}$ such that
\begin{equation}\label{ECRPoisson}
(\na_{\rm NC} u_{\rm ECR}, \na_{\rm NC} v)=(f,v)\text{ for all }v\in V_{\rm ECR}.
\end{equation}

In order to construct a stable finite element method for the Stokes problem, we propose the following
finite element space for the pressure
\begin{equation}
Q_{\rm ECR}:=\{q\in L^2_0(\Omega, \R), q|_K\in P_0(K)\text{ for each }K\in\cT\}.
\end{equation}
The enriched Crouzeix--Raviart element method of Problem \eqref{stokes}
 seeks $(\bu_{\rm ECR}, p_{\rm ECR})\in (V_{\rm ECR})^n \x Q_{\rm ECR}$ such that
\begin{equation}\label{ECRstokes}
\begin{split}
&(\na_{\rm NC} \bu_{\rm ECR},\na_{\rm NC} \bv)+(\div_{\rm NC}\bv,p_{\rm ECR})=(\bf,\bv)\quad \text{ for any }\bv\in (V_{\rm ECR})^n \,,\\
&(\div_{\rm NC}\bu_{\rm ECR}, q)=0 \text{ for any }q\in Q_{\rm ECR}.
\end{split}
\end{equation}
Since $V_{\rm CR}\subset V_{\rm ECR}$, the well-posedness of Problem \eqref{ECRstokes} follows immediately from that for the
Crouzeix--Raviart element method of Problem \eqref{stokes}, see \cite{CrouzeixRaviart(1973)} for more details.

\subsection{The Raviart--Thomas  elements for both the Poisson and Stokes equations}
For the Poisson equation, one famous mixed finite element is the first order Raviart--Thomas element whose shape function
 space is
 $$
 {\rm RT}(K):=(P_0(K))^n+x P_0(K) \text{ for any }K\in\cT.
 $$
 Then the corresponding global finite element space reads
 \begin{equation}
 {\rm RT}(\cT):=\{\tau\in H(\div, \Om, \R^n): \tau|_K\in {\rm RT}(K) \text{ for any }K\in\cT\}.
 \end{equation}
  To get a stable pair of space,  the piecewise constant  space is proposed to
  approximate the displacement, namely,
  \begin{equation}
  U_{\rm  RT}:=\{v\in L^2(\Om, \R): v|_K\in P_0(K)\text{ for any }K\in\cT\}.
  \end{equation}
  The Raviart--Thomas element method of Problem \eqref{MixedPoisson} seeks $(\sigma_{\rm RT}, u_{\rm RT})
  \in {\rm RT}(\cT)\times  U_{\rm  RT}$ such that
 \begin{equation}\label{DiscreteMixedPoisson}
\begin{split}
 &(\sigma_{\rm RT}, \tau)+(u_{\rm RT}, \div \tau)=0\quad \text{ for any }\tau\in {\rm RT}(\cT),\\
 &(\div \sigma_{\rm RT}, v)=(-f, v)\quad \text{ for any }v\in U_{\rm  RT}.
\end{split}
 \end{equation}
 Define
 \begin{equation}
 (\widehat{{\rm RT}}(\cT))^n:=({\rm RT}(\cT))^n \cap  \hat{H}(\div, \Om, \R^{n\times n}).
 \end{equation}
 The Raviart--Thomas element method of Problem \eqref{Mixedstokes} searches  for
 $(\bsi_{\rm RT}, \bu_{\rm RT})\in  (\widehat{{\rm RT}}(\cT))^n\times (U_{\rm  RT})^n$ such that
\begin{equation}\label{DiscreteMixedstokes}
\begin{split}
&(\dev  \bsi_{\rm RT}, \dev\btau )+(\bu_{\rm RT}, \div\btau)=0 \quad \text{ for any }\btau \in (\widehat{{\rm RT}}(\cT))^n\,,\\
&(\div\bsi_{\rm RT}, \bv )=-(\bf, \bv) \text{ for any }\bv\in (U_{\rm  RT})^n.
\end{split}
\end{equation}

\section{Equivalence between the ECR and RT elements}
In this section we assume that both $f$ and $\bf$ are piecewise constant with respect to $\cT$.
\subsection{Equivalence between the ECR and RT elements for the Poisson equation}  Given any $K\in\cT$, let $E_i\,, i=1, 2, \cdots, n+1$,  be its $n-1$ dimensional
 sub-simplices.  Let $\phi_i, i=1, 2, \cdots, n+1$, and $\phi_K$ be basis functions of the shape function space ${\rm ECR}(K)$,
  so that
  \begin{equation}\label{basisfunction}
  \begin{split}
 & \int_{E_i}\phi_jdE=\delta_{i,j}:=\left\{\begin{array}{ll} 1 &\text{ if } i=j\\ 0 & \text{ otherwise}\end{array}\right., \text{ and }\int_K\phi_j dx=0,  i, j=1, \cdots, n+1,\\
 &\int_K \phi_Kdx=1, \text{ and }\int_{E_i}\phi_KdE=0,  i=1, \cdots, n+1.
  \end{split}
  \end{equation}
 See the appendix for the specific expressions.
  \begin{lemma}\label{HDIVProperty} Let $u_{\rm ECR}$  be the solution of Problem \eqref{ECRPoisson}. There holds that
   $$
   \na_{\rm NC}u_{\rm ECR}\in H(\div, \Om, \R^n).
   $$
  \end{lemma}
  \begin{remark} Since $u_{\rm ECR}$ is  nonconforming in the sense that $u_{\rm ECR}\not\in H^1(\Omega, \R)$, it is remarkable that
   $\na_{\rm NC}u_{\rm ECR}$ is $H(\div)$ conforming.
  \end{remark}
  \begin{proof}  Let  $(\sigma_{\rm RT}, u_{\rm RT})$ the solution of Problem \eqref{DiscreteMixedPoisson}.
  Since  $\sigma_{\rm RT}\cdot\nu_E$ are a constant and $\int_E [v]dE=0$ for any $E\in\cE$ and $v\in V_{\rm ECR}$, an integration by parts
 plus  the second equation of \eqref{DiscreteMixedPoisson}  yield
  \begin{equation*}
  (\sigma_{\rm RT}, \na_{\rm NC}v)=\sum\limits_{E\in \cE}\int_E \sigma_{\rm RT}\cdot\nu_E v dE-(\div\sigma_{\rm RT},v) =(f,v).
  \end{equation*}
  This and \eqref{ECRPoisson} lead to
  \begin{equation}
  (\na_{\rm NC}u_{\rm ECR}-\sigma_{\rm RT}, \na_{\rm NC}v)=0 \quad \text{for any } v\in V_{\rm ECR}.
  \end{equation}
  Given $E\in\cE(\Omega)$, let $ v_E\in V_{ECR}$ such that
  $$
  \int_E v_EdE=1, \int_{E^\prime} v_E dE=0 \text{ for any }E^\prime\text{ other than }E, \text{ and }\int_K vdx=0  \text{ for all } K\in\cT.
  $$
  Since $x\cdot\nu_E$ is a constant on $E$, $\na_{\rm NC}u_{\rm ECR}\cdot\nu_E$ is a constant on $E$. Since $\div_{\rm NC}(\na_{\rm NC}u_{\rm ECR}-\sigma_{\rm RT})$ is a piecewise constant function, since both  the average
  $\big\{(\na_{\rm NC}u_{\rm ECR}-\sigma_{\rm RT})\big\}\cdot\nu_E$ and the jump $[(\na_{\rm NC}u_{\rm ECR}-\sigma_{\rm RT})]\cdot\nu_E$
   are a constant on $E$,  an integration by parts derives
  \begin{equation*}
  \begin{split}
  0&=(\na_{\rm NC}u_{\rm ECR}-\sigma_{\rm RT}, \na_{\rm NC}v_E)\\
  &=[(\na_{\rm NC}u_{\rm ECR}-\sigma_{\rm RT})]\cdot\nu_E \int_E v_E dE+\big\{(\na_{\rm NC}u_{\rm NC}-\sigma_{\rm RT})\big\}\cdot\nu_E \int_E [v_E] dE\\
  &=[(\na_{\rm NC}u_{\rm ECR}-\sigma_{\rm RT})]\cdot\nu_E.
   \end{split}
   \end{equation*}
  Hence $\na_{\rm NC}u_{\rm ECR}\in H(\div, \Om, \R^n)$, which completes the proof.
 \end{proof}
 \begin{theorem} Let $u_{\rm ECR}$ and $(\sigma_{\rm RT}, u_{\rm RT})$ be the solutions of problems \eqref{ECRPoisson}
  and \eqref{DiscreteMixedPoisson}, respectively. Then there holds
  $$
\sigma_{\rm RT}=\na_{\rm NC}u_{\rm ECR} \text{ and }  u_{\rm RT}=\Pi_0 u_{\rm ECR},
  $$
  where $\Pi_0$ is the piecewise constant $L^2$ projection operator.
 \end{theorem}
 \begin{proof} By Lemma \ref{HDIVProperty}, we only need to prove that $(\na_{\rm NC}u_{\rm ECR}, \Pi_0 u_{\rm ECR})$
  is the solution of Problem \eqref{DiscreteMixedPoisson}. In fact, given any $\tau\in {\rm RT}(\cT)$,   an integration by parts
  yields
  $$
  (\na_{\rm NC}u_{\rm ECR}, \tau)=-(u_{\rm ECR}, \div \tau)+\sum\limits_{E\in\cE}\int_E [u_{\rm ECR}]\tau\cdot\nu_EdE=-(u_{\rm ECR}, \div \tau).
  $$
 Hence
  $$
  (\na_{\rm NC}u_{\rm ECR}, \tau)+(\Pi_0u_{\rm ECR}, \div \tau)=0,
  $$
  which is the first equation of Problem \eqref{DiscreteMixedPoisson}. To prove the second equation of Problem \eqref{DiscreteMixedPoisson},
   given any $K$, let $v=\phi_K$ in \eqref{ECRPoisson}, an integration by parts leads to
   $$
   (f,\phi_K)=(\na_{\rm NC}u_{\rm ECR}, \na_{\rm NC}\phi_K)
   =-(\div \na_{\rm NC}u_{\rm ECR}, \phi_K)+\sum\limits_{E\subset\partial K}\int_E \na_{\rm NC}u_{\rm ECR}\cdot\nu_E\phi_KdE.
   $$
  Since $\na_{\rm NC}u_{\rm ECR}\cdot\nu_E$ is a constant on $E$ and $\int_E\phi_KdE=0$, this yields
  $$
  \div \na_{\rm NC}u_{\rm ECR}(1, \phi_K)=-f_K(1,\phi_K)\Rightarrow \div \na_{\rm NC}u_{\rm ECR}=-f_K,
  $$
  which completes the proof.
 \end{proof}

 \subsection{Equivalence between the ECR and RT elements for the Stokes equation}
   \begin{lemma}\label{HDIVProperty-Stokes} Let $(\bu_{\rm ECR}, p_{\rm ECR})$  be the solution of Problem \eqref{ECRstokes}. There holds that
   $$
   \na_{\rm NC}\bu_{\rm ECR}+p_{\rm ECR}\id\in H(\div, \Om, \R^{n\times n}).
   $$
  \end{lemma}
  \begin{proof} The proof is actually similar to that of Lemma \ref{HDIVProperty}.
  Let $(\bsi_{\rm RT}, \bu_{\rm RT})$ be the solution of Problem \eqref{DiscreteMixedstokes}. Given any
  $\bv\in (V_{ECR})^n$, it follows from an integration by parts and  the second equation of Problem \eqref{DiscreteMixedstokes}  that
  $$
  (\bsi_{\rm RT}, \na_{\rm NC}\bv)=(\bf, \bv)+\sum\limits_{E\in\cE}\int_{E}\bsi_{\rm RT}\nu_E[\bv] dE=(\bf, \bv).
  $$
  This and the first equation of Problem \eqref{ECRstokes} give
  \begin{equation*}
  (\na_{\rm NC}\bu_{\rm ECR}+p_{\rm ECR}\id-\bsi_{\rm RT}, \na_{\rm NC}\bv)=0\text{ for any }\bv\in (V_{ECR})^n.
  \end{equation*}
  Given any $E\in\cE(\Omega)$,  let $v_E$ be defined as in the proof of  Lemma \ref{HDIVProperty}. Define $\bv_E=(v_E, \cdots, v_E)^T$, this yields
  \begin{equation*}
  \begin{split}
  0=&-(\div_{\rm NC}(\na_{\rm NC}\bu_{\rm ECR}+p_{\rm ECR}\id-\bsi_{\rm RT}), \bv_E)\\
  &+\int_E [\na_{\rm NC}\bu_{\rm ECR}+p_{\rm ECR}\id-\bsi_{\rm RT}]\nu_E\cdot\bv_E dE\\
  =&\int_E [\na_{\rm NC}\bu_{\rm ECR}+p_{\rm ECR}\id-\bsi_{\rm RT}]\nu_E\cdot\bv_E dE\\
  =&[\na_{\rm NC}\bu_{\rm ECR}+p_{\rm ECR}\id-\bsi_{\rm RT}]\nu_E\cdot\int_E\bv_EdE.
  \end{split}
  \end{equation*}
  Since $\bsi_{\rm RT}\in H(\div, \Omega, \R^{n\times n})$,  this proves the desired result.
  \end{proof}
\begin{theorem} Let $(\bu_{\rm ECR}, p_{\rm ECR})$  be the solution of Problem \eqref{ECRstokes},
and let $(\bsi_{\rm RT}, \bu_{\rm RT})$ be the solution of Problem \eqref{DiscreteMixedstokes}. Then there holds that
\begin{equation*}
 \bsi_{\rm RT}=\na_{\rm NC} \bu_{\rm ECR}+p_{\rm ECR}\id \text{ and }\bu_{\rm RT}=\Pi_0\bu_{\rm ECR}+\mathcal{L}u_{\rm ECR},
\end{equation*}
where $\mathcal{L}u_{ECR}\in (U_{\rm RT})^n $ is defined by
 $$
 (\mathcal{L}u_{\rm ECR}, \div \btau)=( \div_{\rm NC} \bu_{\rm ECR}, \tr\btau/n)\text{ for any }\btau\in (\widehat{{\rm RT}}(\cT))^n.
 $$
 \begin{remark}Since $\Pi_0  \div_{\rm NC} \bu_{\rm ECR}=0$ and $\div \btau=0$ implies that $\btau$ is a piecewise constant matrix--valued
  function, the operator $\mathcal{L}$ is well--defined.
 \end{remark}
\end{theorem}
\begin{proof} We prove that $(\na_{\rm NC} \bu_{\rm ECR}+p_{\rm ECR}\id, \Pi_0\bu_{\rm ECR}+\mathcal{L}u_{\rm ECR})$ is the solution of Problem \eqref{DiscreteMixedstokes}. We start with a simple but important property of the deviatoric operator $\dev$ as follows
$$
(\dev \bsi, \dev\btau)=(\bsi, \dev\btau)=(\dev\bsi, \btau) \text{ for any }\bsi, \btau\in H(\div, \Omega, \R^{n\times n}).
$$
 Hence, any $\btau\in (\widehat{{\rm RT}}(\cT))^n$ admits the following decomposition:
\begin{equation}\label{decom}
\begin{split}
&(\dev(\na_{\rm NC} \bu_{\rm ECR}+p_{\rm ECR}\id), \dev\btau)=(\dev \na_{\rm NC} \bu_{\rm ECR}, \btau)=( \na_{\rm NC} \bu_{\rm ECR},\dev \btau)\\
&=( \na_{\rm NC} \bu_{\rm ECR}, \btau)-( \div_{\rm NC} \bu_{\rm ECR}, \tr\btau/n).
\end{split}
\end{equation}
After integrating by parts, the first term on the right--hand side of \eqref{decom} can be rewritten as
$$
( \na_{\rm NC} \bu_{\rm ECR}, \btau)=-(\bu_{\rm ECR}, \div\btau)=-(\Pi_0\bu_{\rm ECR}, \div\btau)
$$
since $\sum\limits_{E\in\cE}\int_E[\bu_{\rm ECR}]\cdot \btau\nu_EdE=0$.  This proves that
$$
(\dev(\na_{\rm NC} \bu_{\rm ECR}+p_{\rm ECR}\id), \dev\btau)+(\Pi_0\bu_{\rm ECR}+\mathcal{L}u_{\rm ECR}, \div\btau)=0,
$$
which is the first equation of Problem \eqref{DiscreteMixedstokes}.  Given any $K$,  define $\bv_K=(\phi_K, \cdots, \phi_K)$. Let $\bv=\bv_K$
  in \eqref{Mixedstokes}.  After integrating by parts, we derive as
  $$
-(\bf, \bv_K)=(\div (\na_{\rm NC} \bu_{\rm ECR}+p_{\rm ECR}\id), \bv_K)\Rightarrow \div (\na_{\rm NC} \bu_{\rm ECR}+p_{\rm ECR}\id)=-\bf.
  $$
  Since it is obvious that $\int_{\Om} \div_{\rm NC} \bu_{\rm ECR}dx=0$,  $\na_{\rm NC} \bu_{\rm ECR}+p_{\rm ECR}\id\in (\widehat{{\rm RT}}(\cT))^n$.
  This completes the proof.
\end{proof}

\subsection{Comments on the Poisson problem with the pure Neumann boundary}
 Given a bounded domain $\Omega\subset \R^n$ with $n\geq 2$ and $f\in L^2(\Om, \R)$,  the Poisson model  problem
 with the pure Neumann boundary condition finds $u\in \hat{H}^1(\Om, \R):=\{v\in H^1(\Omega, \R): \int_{\Omega}vdx=0\}$ such that
\begin{equation}\label{NeumannPoisson}
(\na u, \na v)=(f,v)+<g,v>\quad \text{ for all }v\in \hat{H}^1(\Om, \R).
\end{equation}
where $g|_{\partial \Omega}:=\frac{\partial u}{\partial\nu}|_{\partial \Om}\in H^{-1/2}(\partial\Om,\R)$.
Suppose that $(f,v)+<g,v>=0$, this problem admits a unique solution.
 For this problem,  the  equivalent mixed  formulation seeks $(\sigma, u)\in H_g(\div, \R^n)\times L^2(\Omega,\R)$ such that
 \begin{equation}\label{NeumannMixedPoisson}
\begin{split}
 &(\sigma, \tau)+(u, \div \tau)=0\quad \text{ for any }\tau\in H_0(\div,\Omega,\R^n),\\
 &(\div \sigma, v)=(-f, v)\quad \text{ for any }v\in L^2(\Omega, \R).
\end{split}
 \end{equation}
 Here
 $$
 H_D(\div, \R^n)=\{\tau\in H(\div,\Omega,\R^n): \tau\cdot\nu=D\text{ on }\partial\Omega\} \text{ with } D=0 \text{ or }g.
 $$
 Suppose that both $f$ and $g$ are a piecewise constant function. Then the result in \eqref{ECR-RT} holds equally for this case.
 Since the space for the CR element is a subspace of the ECR element,  this implies that the CR element can not be  equal to the RT element.
 In fact, for two dimensions,  let the  exact solution of Problem \eqref{NeumannPoisson} be $u=x_1^2+x_2^2$, which yields that
 $f=-4$ and $g$ is a piecewise constant on a polygonal domain.  For this problem, the RT element gives the exact solution while the error
 of the CR  element has the following lower bound
 $$
 \beta h\leq \|\na_{\rm NC}(u-u_{\rm CR})\|
 $$
 for some positive constant $\beta$ and the meshsize $h$ of the domain, see \cite{HuHuangLin2010} for more details of proof.

\section{Equivalence between the ECR and RT elements for eigenvalue problem}
First we introduce the eigenvalue problem for the Laplace operator and the finite element method in this section. The eigenvalue problem finds $(\lam, u)\in\R\x H^1_0(\Om, \R)$ such that
\begin{equation}\label{eigen}
\begin{split}
(\nabla u, \nabla v) =\lam (u, v) \text{ for any }v\in L^2(\Omega, \R)\text{ and }\|u\|=1.
\end{split}
\end{equation}
By introducing an auxiliary variable $\sigma:=\nabla u$,  the problem can be formulated as the following equivalent
mixed problem which seeks $(\lam,\sigma, u)\in \R\x H(\div,\Omega,\R^n)\times L^2(\Omega, \R)$ such that
 \begin{equation}\label{MixedEigen}
\begin{split}
 &(\sigma, \tau)+(u, \div \tau)=0\quad \text{ for any }\tau\in H(\div,\Omega, \R^n),\\
 &(\div \sigma, v)=-\lam(u, v)\quad \text{ for any }v\in L^2(\Omega, \R)\text{ and }\|u\|=1.
\end{split}
 \end{equation}
The ECR element method of \eqref{eigen} seeks $(\lam_{\rm ECR}, u_{\rm ECR})\in \R\x V_{\rm ECR}$ such that
\begin{equation}\label{discreteeigentotal}
\begin{split}
(\nabla_{\rm NC}u_{\rm ECR}, \nabla_{\rm NC}v)&=\lam_{\rm ECR} (u_{\rm ECR},v) \text{ for any }
v\in V_{\rm ECR} \text{ and }  \|u_{\rm ECR}\|=1.
\end{split}
\end{equation}
The RT element method of Problem \eqref{MixedEigen} seeks $(\lam_{\rm RT},\sigma_{\rm RT},u_{\rm RT})\in \R\x{\rm RT}(\cT)\x U_{\rm RT}$ such that
 \begin{equation}\label{dicreteMixedEigen}
\begin{split}
 &(\sigma_{\rm RT}, \tau)+(u_{\rm RT}, \div \tau)=0\quad \text{ for any }\tau\in {\rm RT}(\cT),\\
 &(\div \sigma_{\rm RT}, v)=-\lam_{\rm RT}(u_{\rm RT}, v)\quad \text{ for any }v\in U_{\rm RT} \text{ and }\|u_{\rm RT}\|=1.
\end{split}
 \end{equation}

Assume, for simplicity, we only consider the case of $\lam$ is an eigenvalue of multiplicity 1. We define $T$ as the inverse operator of continuous problem, i.e. for any $f\in L^2(\Omega,\R)$, $Tf=u_f\in H^1_0(\Omega,\R)$, where $u_f$ satisfies the Poisson equation \eqref{Poisson}, i.e.
\begin{equation}
\label{operatorT}
  (\nabla u_f,\nabla v)=(f,v)\quad\text{ for any }v\in H^1_0(\Omega,\R).
\end{equation}
Generally speaking, the regularity of $u_f$ depends on, among others, regularities of $f$ and the shape of the domain $\Omega$. To fix the main idea and therefore avoid too technical notation, throughout the remaining paper, without loss of generality, assume that $u_f\in H^1_0(\Omega,\R)\cap H^{1+s}(\Omega,\R)$ with $0<s\leq1$ in the sense that
\begin{equation}\label{regularity}
\|u_f\|_{H^{1+s}(\Omega,\R)}\lesssim\|f\|.
\end{equation}
Here and throughout the paper, the inequality $A\lesssim B$ replaces $A\leq CB$ with some multiplicative mesh--size independent constant $C>0$ that depends on the domain $\Omega$, the shape of element, and possibly the eigenvalue $\lam$.

It follows from the theory  of nonconforming eigenvalue approximation \cite{HuHuangLin2010} and known a priori estimate that,
\begin{equation}
\label{errorECR}
  |\lam-\lam_{\rm ECR}|+\|u-u_{\rm ECR}\|+h^s\|\nabla_{\rm NC}(u-u_{\rm ECR})\|\lesssim h^{2s}\|u\|_{H^{1+s}(\Omega,\R)}
\end{equation}
and the theory of mixed eigenvalue approximation \cite{MercierOsborn1981} that
\begin{equation}
\label{errorRT}
  |\lam-\lam_{\rm RT}|+h^s(\|\sigma-p_{\rm RT}\|+\|u-u_{\rm RT}\|)\lesssim h^{2s}\|u\|_{H^{1+s}(\Omega,\R)}.
\end{equation}
Using \eqref{regularity}, the bound for the eigenvalue $\lam\lesssim1$ and $\|u\|=1$, there holds that
\begin{equation*}
\|u\|_{H^{1+s}(\Omega,\R)}\lesssim\|\lam u\|\lesssim1.
\end{equation*}

To analyze the equivalence, we introduce the following discrete problem: Find $\phi_{\rm ECR}\in V_{\rm ECR}$ such that
\begin{equation}\label{equiEgen}
(\nabla_{\rm NC}\phi_{\rm ECR},\nabla_{\rm NC} v)=\lam_{\rm RT}(\Pi_0\phi_{\rm ECR},v)\text{ for any }v\in V_{\rm ECR}.
\end{equation}
It follows from Theorem 3.3 that Problem \eqref{dicreteMixedEigen} is equivalent to \eqref{equiEgen} in the sense that they have the same eigenvalues $\lam_{\rm RT}$ and the eigenvectors are related by $\sigma_{\rm RT}=\nabla_{\rm NC}\phi_{\rm ECR}$ and $u_{\rm RT}=\Pi_0\phi_{\rm ECR}$.

Similar to the analysis in \cite{Duran1999}, applying to Problem \eqref{equiEgen} the general theory developed for example in \cite{BabuskaOsborn1991} we can prove that
\begin{equation}
\label{error}
\| u-\bar{\phi}_{\rm ECR}\|\lesssim h^{2s},
\end{equation}
where $\bar{\phi}_{\rm ECR}=\phi_{\rm ECR}/\|\phi_{\rm ECR}\|$.
To present it clearly, we follow a similar argument in \cite{Duran1999} and give the proof of \eqref{error}. Let $T_h$ be defined as the inverse operators of the following discrete problem, i.e., for $f\in L^2(\Omega,\R)$, $T_hf=w_h\in V_{\rm ECR}$ where  $w_h$ satisfies

\begin{equation}
\label{operatorTh}
  (\nabla_{\rm NC} w_h,\nabla_{\rm NC} v)=(\Pi_0f,v)\text{ for any }v\in V_{\rm ECR}.
\end{equation}
Let $\mathbb{E}$ denote the eigenspace corresponding to $\lam$. We have the following two results.
\begin{lemma}
Suppose $T$ is defined in \eqref{operatorT} and $T_h$ is defined in \eqref{operatorTh}. Then,
\begin{equation*}
\|(T-T_h)|_{\mathbb{E}}\|_{\mathcal{L}(L^2,L^2)}\lesssim h^{2s}.
\end{equation*}
\end{lemma}
\begin{proof}
We have to show that
\begin{equation*}
  \|Tf-T_hf\|\lesssim h^{2s}\|f\|\text{ for any }f\in \mathbb{E}.
\end{equation*}
Let $u_f=Tf$, $u_{\Pi_0f}=T(\Pi_0f)$ and $w_h=T_hf$. First a standard argument for nonconforming finite element methods, see, for instance, \cite{BrennerScott1996}, proves
\begin{equation}\label{4.1}
  \|T(\Pi_0f)-T_hf\|=\|u_{\Pi_0f}-w_h\|\lesssim h^{2s}\|\Pi_0f\|\lesssim h^{2s}\|f\|.
\end{equation}
Let $e=u_f-u_{\Pi_0f}$ and $r\in H^1_0(\Omega,\R)$ be the solution of $-\Delta r=e$. Then a standard duality argument gives,
\begin{equation}\label{lemma1}
\begin{split}
  (e,e)&=(\nabla e,\nabla r)=(f-\Pi_0f,r)\\
  &=(f-\Pi_0f,r-\Pi_0r).
  \end{split}
\end{equation}
Hence, the property of piecewise constant $L^2$ projection $\Pi_0$ implies that
$$
\|e\|\lesssim h^2\|\nabla f\|.
$$
Since $f\in \mathbb{E}$, there exists a constant C depending on $\lam$ such that $\|\nabla f\|\leq C\|f\|$ and so
$$
\|e\|\lesssim h^2\|f\|.
$$
This and \eqref{4.1} complete the proof.
\end{proof}
\begin{lemma}

The sequence $\{T_h\}_h$ converges uniformly to $T$ in $\mathcal{L}(L^2,L^2)$ as $h$ goes to 0.
\end{lemma}
\begin{proof}
We show that for all $f\in L^2(\Omega,\R)$ we have
$$
\|Tf-T_hf\|\lesssim h^{\min\{2s,1\}}\|f\|.
$$
The proof follows the same lines as the previous lemma. The fact that $f$ belonged to the eigenspace $\mathbb{E}$ was used only once to estimate \eqref{lemma1} with the desired order. When $f$ is taken in $L^2(\Omega,\R)$ we can only obtain from the following bound, using similar arguments as before
\begin{equation*}
\begin{split}
  (e,e)&=(\nabla e,\nabla r)=(f-\Pi_0f,r)\\
  &=(f-\Pi_0f,r-\Pi_0r)\\
  &\lesssim h\|f-\Pi_0f\|\|\nabla r\|\\
  &\lesssim h\|f\|\|e\|.
  \end{split}
\end{equation*}
This and \eqref{4.1} imply the desired order.
\end{proof}
Since the sequence of operators $\{T_h\}_h$ converges uniformly to $T$ in $\mathcal{L}(L^2,L^2)$, well-known results in the theory of spectral approximation yield the following error estimate for eigenvectors, see e.g. \cite{BabuskaOsborn1991}
\begin{equation}\label{error2}
\|u-\bar{\phi}_{\rm ECR}\|\lesssim\|(T-T_h)|_{\mathbb{E}}\|_{\mathcal{L}(L^2,L^2)}.
\end{equation}
Then \eqref{error} is a consequence of \eqref{error2} and Lemma 4.1. In fact,
 \begin{equation*}
 \begin{split}
  \|\phi_{\rm ECR}\|^2&=\|(I-\Pi_0)\phi_{\rm ECR}\|^2+\|\Pi_0\phi_{\rm ECR}\|^2\\
  &=1+\|(I-\Pi_0)\phi_{\rm ECR}\|^2.\\
  \end{split}
\end{equation*}
This and the property of piecewise constant $L^2$ projection $\Pi_0$ yield $0\leq\|\phi_{\rm ECR}\|-1\lesssim \lam_{\rm RT}h^2$, and so $\phi_{\rm ECR}$ satisfies
\begin{equation}\label{errorRT2}
\| u-\phi_{\rm ECR}\|\lesssim h^{2s}.
\end{equation}

The equivalence result for the errors of the eigenfunction approximations is presented as follows.
\begin{theorem}
For sufficiently small $h\ll1$, the discrete eigenfunctions $u_{\rm ECR}$ and $\sigma_{\rm RT}$ satisfy
$$
\|\nabla u-\nabla_{\rm NC}u_{\rm ECR}\|=\|\nabla u-\sigma_{\rm RT}\|+h.o.t.
$$
\end{theorem}
\begin{proof}
Using \eqref{equiEgen} and some elementary manipulation yield
\begin{equation}
\label{errorECRRT}
  \begin{split}
 \|\nabla_{\rm NC}u_{\rm ECR}-\sigma_{\rm RT}\|^2=&\|\nabla_{\rm NC}u_{\rm ECR}-\nabla_{\rm NC}\phi_{\rm ECR}\|^2\\
=&(\lam_{\rm ECR}u_{\rm ECR}-\lam_{\rm RT}\Pi_0\phi_{\rm ECR},u_{\rm ECR}-\phi_{\rm ECR})\\
=&(\lam_{\rm ECR}u_{\rm ECR}-\lam_{\rm RT}\phi_{\rm ECR},u_{\rm ECR}-\phi_{\rm ECR})\\
&+\lam_{\rm RT}((I-\Pi_0)\phi_{\rm ECR},(I-\Pi_0)(u_{\rm ECR}-\phi_{\rm ECR}))\\
\lesssim&\|\lam_{\rm ECR}u_{\rm ECR}-\lam_{\rm RT}\phi_{\rm ECR}\|\|u_{\rm ECR}-\phi_{\rm ECR}\|\\
&+\lam_{\rm RT}h^2\|\nabla_{\rm NC}\phi_{\rm ECR}\|\|\nabla_{\rm NC}u_{\rm ECR}-\sigma_{\rm RT}\|.
\end{split}
\end{equation}
The bound for the eigenvalues $\lam_{\rm ECR}\lesssim1$,$\lam_{\rm RT}\lesssim1$ and the normalisation $\|u_{\rm ECR}\|=1$ yield
\begin{equation*}
\|\lam_{\rm ECR}u_{\rm ECR}-\lam_{\rm RT}\phi_{\rm ECR}\|\lesssim|\lam_{\rm ECR}-\lam_{\rm RT}|+\|u_{\rm ECR}-\phi_{\rm ECR}\|.
\end{equation*}
Therefore, the Young inequalities, \eqref{errorECR} and \eqref{errorRT2} control the first term in \eqref{errorECRRT} as
\begin{equation*}
\begin{split}
\|\lam_{\rm ECR}&u_{\rm ECR}-\lam_{\rm RT}\phi_{\rm ECR}\|\|u_{\rm ECR}-\phi_{\rm ECR}\|\\
&\lesssim|\lam_{\rm ECR}-\lam_{\rm RT}|^2+\|u_{\rm ECR}-\phi_{\rm ECR}\|^2\lesssim h^{4s}.
\end{split}
\end{equation*}
The last term in \eqref{errorECRRT} can be absorbed. Hence it yields that
\begin{equation*}
 \|\nabla_{\rm NC}u_{\rm ECR}-\sigma_{\rm RT}\|\lesssim h^{2s},
\end{equation*}
which is a high order term.
\end{proof}
\section{Numerical results}
In this section, we present some numerical results, which show that ECR elements have some good numerical properties.
\subsection{Poisson problem}
We consider the poisson problem \eqref{Poisson}. Define the bubble function space
\begin{equation*}
  B_{\rm ECR}:=\{v\in L^2(\Omega,\R): v|_K\in \sspan\{\phi_K\},\forall K\in\mathcal{T}\},
\end{equation*}
where $\phi_K$ is defined in \eqref{basisfunction}.
For any $v\in V_{\rm ECR}$, $\Pi v\in V_{\rm CR}$ is given by
\begin{equation*}
  \int_E\Pi vds=\int_Ev ds\text{ for all }E\in\mathcal{E}.
\end{equation*}
Hence $v-\Pi v$ has vanishing average on each $E$ and $v-\Pi v\in B_{\rm ECR}$. Let $u_{\rm ECR}$ be the solution to the discrete problem by the ECR element, then $u_{\rm ECR}$ can be written as $u_{\rm ECR}=\Pi u_{\rm ECR}+u^b$, where $\Pi u_{\rm ECR}\in V_{CR}$ and $u^b\in B_{\rm ECR}$. In \eqref{ECRPoisson}, we choose
\begin{equation*}
  v=\begin{cases}
  \phi_K\quad& x\in K,\\
  0&x\not\in K.
  \end{cases}
\end{equation*}
This gives
$$
(\nabla u_{\rm ECR},\nabla \phi_K)_{L^2(K)}=(f,\phi_K)_{L^2(K)}.
$$
Since $\int_{E_i}\phi_KdE=0, i=1,\cdots,n+1$, an integration by parts leads to the following important orthogonality:
\begin{equation}\label{functionOrthogonality}
(\nabla \Pi u_{\rm ECR},\nabla\phi_K)_{L^2(K)}=0.
\end{equation}
This leads to
\begin{equation}\label{bubblesoving}
(\nabla u^b,\nabla\phi_K)_{L^2(K)}=(f,\phi_K)_{L^2(K)}\quad\forall K\in\mathcal{T},
\end{equation}
and
\begin{equation}\label{CRproblem}
  (\nabla_{\rm NC} \Pi u_{\rm ECR},\nabla_{\rm NC} v)=(f,v)\text{ for all }v\in V_{\rm CR}.
\end{equation}
Consequently, $\Pi u_{ECR}$ is the solution to the discrete problem by the CR element.
Hence we can solve the ECR element by solving \eqref{bubblesoving} on each $K$ and \eqref{CRproblem} for the CR element, respectively.
\begin{remark}For general second order elliptic problems: Find $u\in H^1_0(\Omega,\R)$ such that
\begin{equation*}
  (A\nabla u,\nabla v)=(f,v)\quad\text{ for all }v\in H^1_0(\Omega,\R),
\end{equation*}
when $A$ is a piecewise constant tensor-valued function, a similar  orthogonality of \eqref{functionOrthogonality} still holds
\begin{equation}\label{Orthogonalitysecond}
  (A\nabla \Pi u_{\rm ECR},\nabla\phi_K)_{L^2(K)}=0\quad\forall K\in\mathcal{T}.
\end{equation}
Hence, we can still use the same technique to solve the ECR element. For the more general case, the orthogonality  \eqref{Orthogonalitysecond} does not hold. However, $u^b$ can be eliminated a prior by a static condensation procedure.
\end{remark}
We compute two examples which compare the errors of  the ECR and CR elements. The first example takes $\Omega=(0,1)^2$ and the exact solution $u(x,y)=\sin(\pi x)\sin(\pi y)$; the second takes $\Omega=(0,1)^3$ and the exact solution $u(x,y,z)=\sin(\pi x)\sin(\pi y)\sin(\pi z)$. Both comparisons are illustrated in \reffig{h1error}, which indicates that $\|\nabla_{\rm NC}(u-u_{\rm ECR})\|$ is smaller than $\|\nabla_{\rm NC}(u-u_{\rm CR})\|$.
\begin{figure}[!ht]
\centering
  \subfigure[$u(x,y)=\sin(\pi x)\sin(\pi y)$]{\includegraphics[width=9cm,height=5cm]{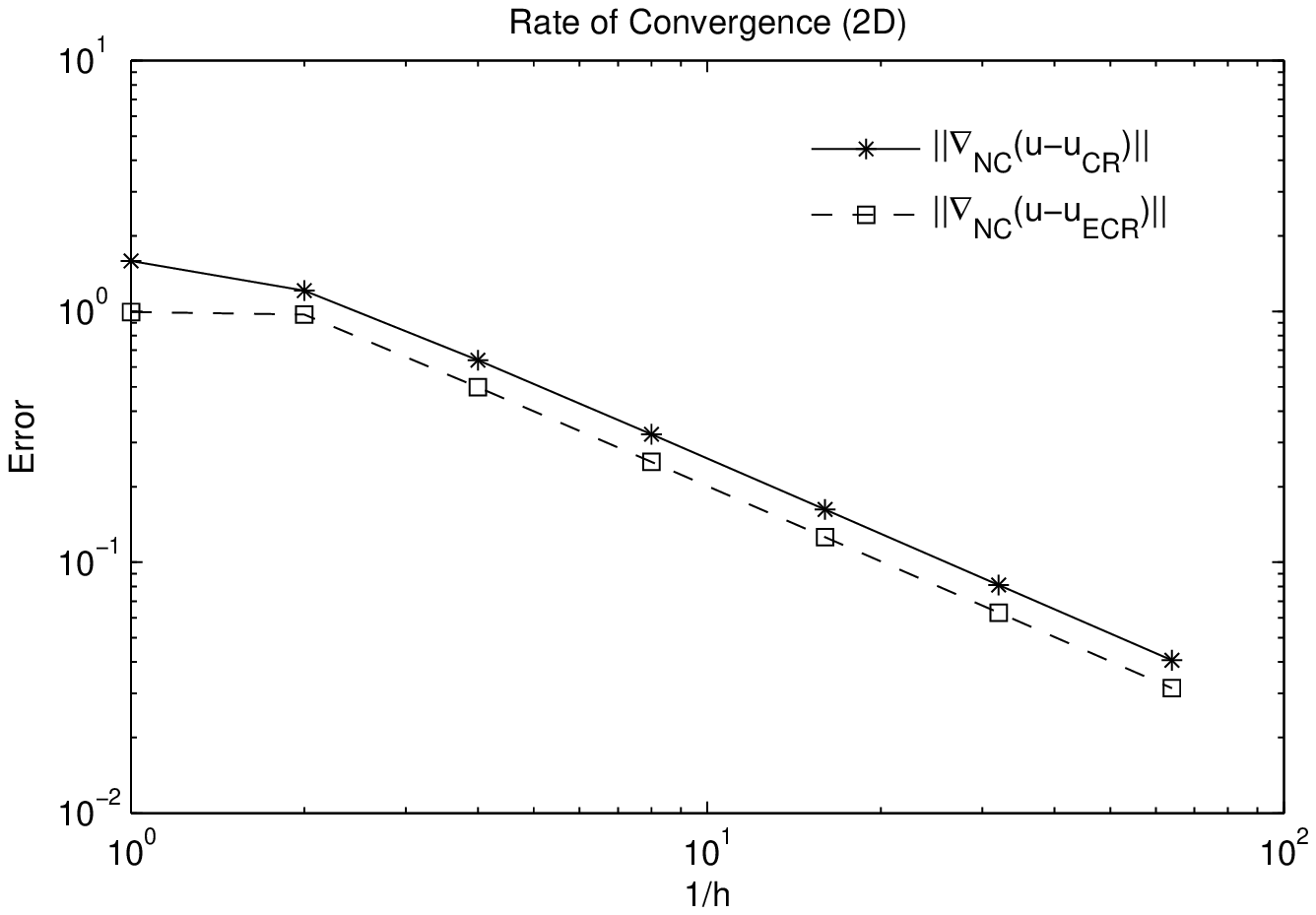}}
  \label{Fig.sub.1}
  \subfigure[$u(x,y,z)=\sin(\pi x)\sin(\pi y)\sin(\pi z)$]{\includegraphics[width=9cm,height=5cm]{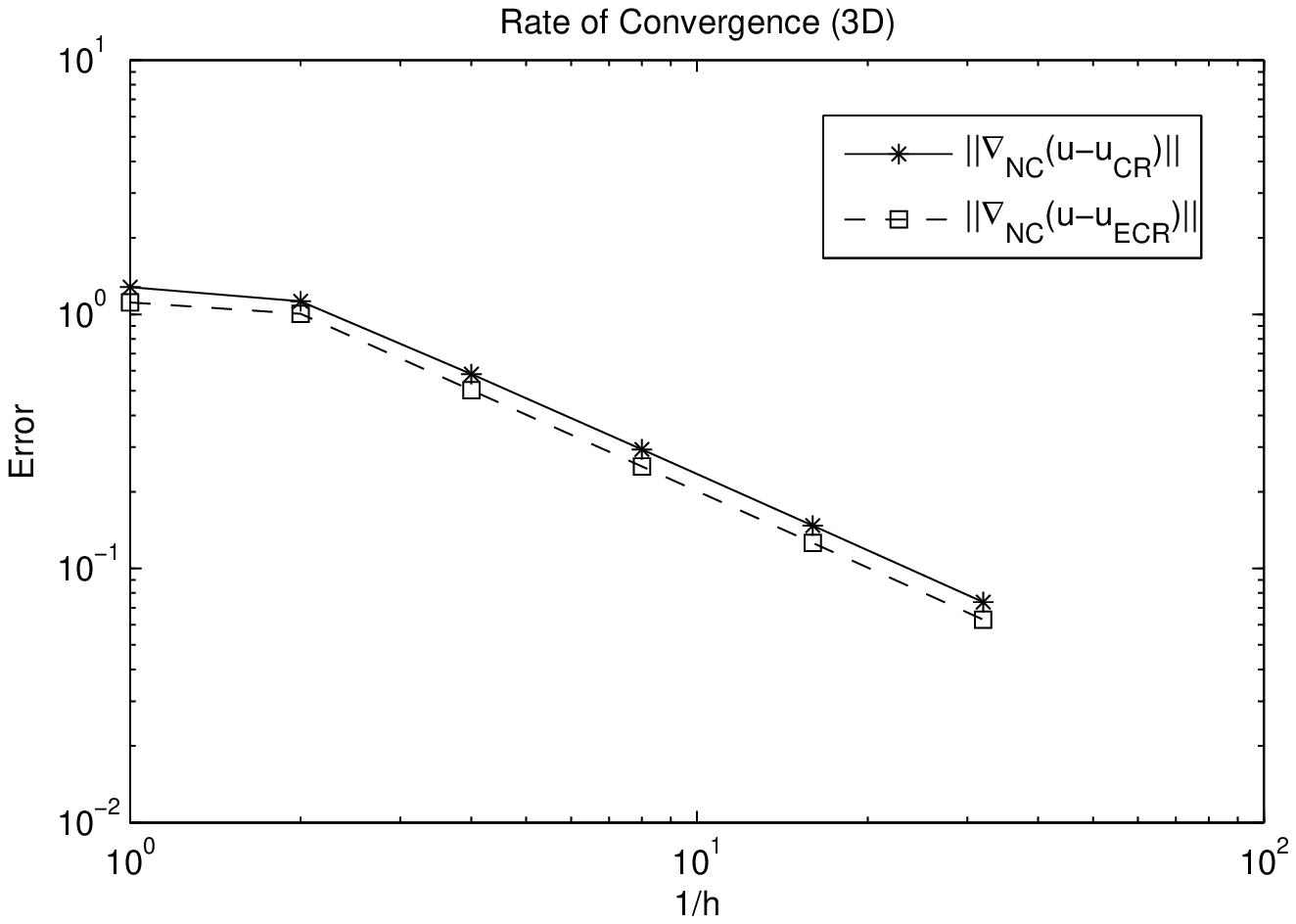}}
  \label{Fig.sub.2}
 \caption{}\label{h1error}
\end{figure}
\subsection{Eigenvalue problem}
We consider the eigenvalue problem \eqref{eigen}. Since $V_{\rm CR}\subset V_{\rm ECR}$, the eigenvalues produced by the ECR element are smaller than those by the CR element. When the meshsize is small enough, the ECR element has been proved to produce lower bounds for eigenvalues, see \cite{HuHuangLin2010}. When eigenfunctions are singular, the CR element provides lower bounds for eigenvalues, see \cite{Armentano}; under some mesh conditions, it also produces lower bounds for eigenvalues when eigenfunctions are smooth, see \cite{HuHuangShen2014}.

On the coarse triangulation of the
square domain $\Omega=(0,1)^2$ from \reffig{trianglulation}, the CR element produces a upper bound $\lam_{\rm CR}=24$ for the first eigenvalue $\lam=2\pi^2\approx19.7392$ of the Laplace operator, while the ECR element gives a lower bound $\lam_{\rm ECR}=17.1429$.

\begin{figure}[!ht]
\centering
  \includegraphics[width=10cm]{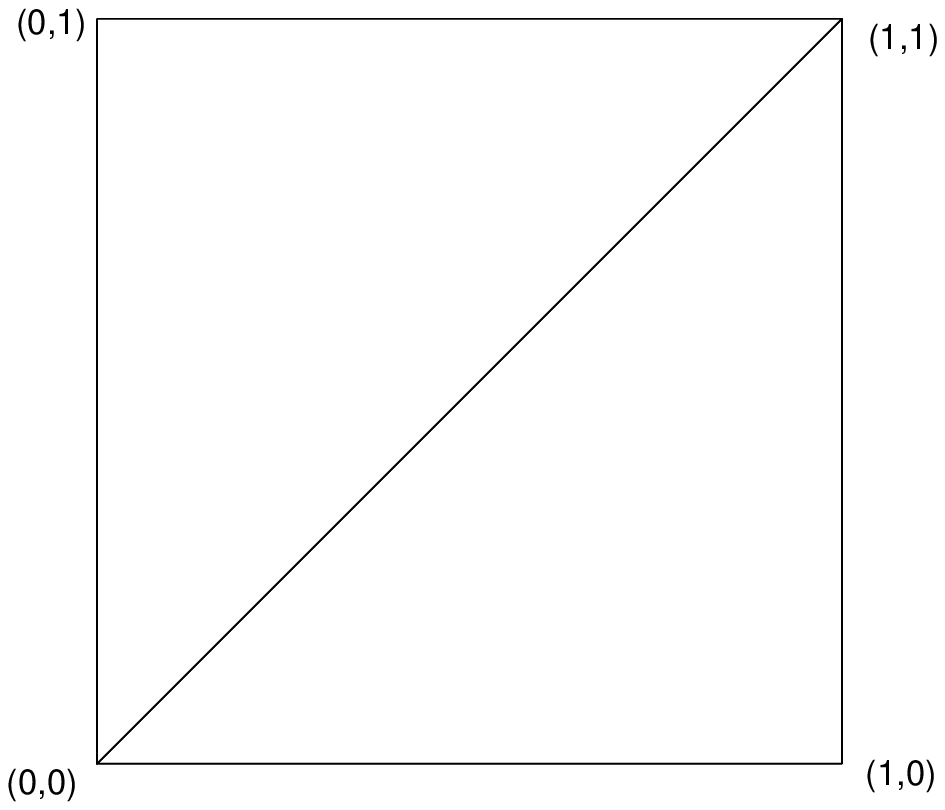}\\
 \caption{}\label{trianglulation}
\end{figure}
\begin{appendix}
\section{Basis Functions and Convergence Analysis of the ECR Element}
For any $K\in\mathcal{T}$, we give the basis functions of the shape function space ${\rm ECR}(K)$.
Suppose the coordinate of  the centroid ${\rm mid}(K)$ is $(M_1,M_2\cdots,M_n)$. The vertices of $K$ are denoted by $a_i,1\leq i\leq n+1$ and the barycentric coordinates by $\lam_1,\lam_2,\cdots,\lam_{n+1}$. Let $H=\sum_{i<j}|a_i-a_j|^2$, then the basis functions are as follows
\begin{equation*}
\begin{split}
&\phi_K = \frac{n+2}{2}-\frac{n(n+1)^2(n+2)}{2H}\sum^{n}_{i=1}(x_i-M_i)^2,\\
&\phi_j=1-n\lam_j-\frac{1}{n+1}\phi_K,\quad 1\leq j\leq n+1.
\end{split}
\end{equation*}

For any $v\in V_{\rm ECR}$, by the definition of $V_{\rm ECR}$ in \eqref{ECR}, $\int_{E}[v]dE=0\text{ for all }E\in\mathcal{E}(\Omega)$ and $\int_EvdE=0\text{ for all }E\in\mathcal{E}(\partial\Omega)$. From the theory of \cite{HuMa2014}, there holds that
\begin{equation*}
  \|\nabla_{\rm NC}(u-u_{\rm ECR})\|\lesssim\|\nabla u-\Pi_0\nabla u\|+osc(f),
\end{equation*}
where $$osc(f)=\left(\sum_{K\in\mathcal{T}}h_K^2\big[\inf_{\bar{f}\in P_r(K)}\|f-\bar{f}\|^2_{0,K}\big]\right)^{1/2}$$
$r\geq 0$ is arbitrary. The convergence of the ECR element follows immediately.
\end{appendix}

\end{document}